\documentclass[11pt]{amsart}

\usepackage{amsmath, amscd, amssymb}
\usepackage[frame,cmtip,arrow,matrix,line,graph,curve]{xy}
\usepackage{graphpap, color}
\usepackage[mathscr]{eucal}
\usepackage{cancel}
\usepackage{verbatim}
 
\numberwithin{equation}{section}

\def\sO{{\mathscr O}}

\newcommand{\CC}{\mathbb{C}}

\newcommand{\PP}{\mathbb{P}}
\newcommand{\QQ}{\mathbb{Q}}

\newcommand{\cal}{\mathcal}

\def\cC{{\cal C}}

\def\cE{{\cal E}}
\def\cF{{\cal F}}
\def\cG{{\cal G}}

\def\cV{{\cal V}}

\def\cI{{\cal I}}

\def\fN{\mathfrak{N}}

\def\loc{\mathrm{loc}}

\def\and{\quad{\rm and}\quad}
\def\lra{\longrightarrow }
\def\mapright#1{\,\smash{\mathop{\lra}\limits^{#1}}\,}
\def\mapleft#1{\,\smash{\mathop{\longleftarrow}\limits^{#1}}\,}

\newtheorem{prop}{Proposition}[section]
\newtheorem{theo}[prop]{Theorem}
\newtheorem{lemm}[prop]{Lemma}

\newtheorem{exam}[prop]{Example}
\newtheorem{ques}[prop]{Question}
\newtheorem{defi}[prop]{Definition}

\newtheorem{assu}[prop]{Assumption}

\def\beq{\begin{equation}}
\def\eeq{\end{equation}}

\def\virt{^{\mathrm{vir}} }

\def\Spec{\mathrm{Spec} }

\def\bbL{\mathbb{L} }
\def\DM{Deligne-Mumford }

\def\lalp{_\alpha }
\def\albe{{\alpha\beta} }
\def\fX{\mathfrak{X} }
\def\bfc{\mathbf{c} }

\title[Localizing virtual cycles]{Localizing virtual fundamental cycles for semi-perfect obstruction theories}

\author{Young-Hoon Kiem}
\address{Department of Mathematics and Research Institute
of Mathematics, Seoul National University, Seoul 08826, Korea}
\email{kiem@snu.ac.kr}

\thanks{Partially supported by Samsung Science and Technology Foundation.}

\date{October 31, 2016}

\begin{document}
\begin{abstract} Recently H.-L. Chang and J. Li generalized the theory of virtual fundamental class to the setting of semi-perfect obstruction theory. A semi-perfect obstruction theory requires only the \emph{local} existence of a perfect obstruction theory with compatibility conditions. 
In this paper, we generalize the torus localization of Graber-Pandharipande \cite{GrPa}, the cosection localization \cite{KLc} and their combination \cite{CKL}, to the setting of semi-perfect obstruction theory. As an application, we show that the Jiang-Thomas theory \cite{JiTh} of virtual signed Euler characteristic works without the technical quasi-smoothness assumption from derived algebraic geometry.\end{abstract}
\maketitle  

\section{Introduction}\label{sec1}

The theory of virtual fundamental class was invented 
in 1995 by Li-Tian \cite{LiTi} and Behrend-Fantechi \cite{BeFa}
in order to provide a rigorous algebro-geometric theory of the Gromov-Witten invariant.
Since then, the virtual fundamental class has played a key role
in defining important invariants like Donaldson-Thomas and Pandharipande-Thomas invariants, 
as integrals on the virtual fundamental cycles on suitable moduli spaces. Each \DM stack $X$ 
has the intrinsic normal cone $\bfc_X$ canonically embedded into its abelian hull $\fN_X=h^1/h^0(\bbL_X^\vee)$ where $\bbL_X$ denotes the cotangent complex of $X$ (cf. \cite{BeFa}). A perfect obstruction theory $\phi:E\to \bbL_X$ gives us an embedding of
$\fN_X$ into the vector bundle stack $\cE_X=h^1/h^0(E^\vee)$ and the virtual fundamental class $[X]\virt$ is defined by applying the Gysin map $0^!_{\cE_X}$ to $[\bfc_X]$.
   
 A few effective techniques to handle virtual fundamental classes were discovered 
during the past two decades, such as
the torus localization of Graber-Pandharipande \cite{GrPa}, the degeneration method of J. Li \cite{LiD} and the cosection localization \cite{KLc}. Often combining these techniques turns out to be quite effective.
In \cite{CKL}, it was proved that the torus localization works for the cosection localized virtual fundamental classes and this combined localization turned out to be quite useful for the Landau-Ginzburg/Calabi-Yau correspondence \cite{CLLL}. 

Recently thanks to improved knowledge on derived categories and wall crossing, there arose a demand to handle more moduli spaces of derived category objects on a smooth projective variety. Unfortunately it is often hard to find perfect obstruction theories on moduli spaces of derived category objects and hence virtual fundamental classes were not readily available for them. 
In \cite{CLs}, H.-L. Chang and J. Li constructed the virtual fundamental class under a weaker condition than the existence of a perfect obstruction theory. They require only the \emph{local} existence of compatible perfect obstruction theories and it was proved in \cite{CLs} that these local theories are sufficient to give the virtual fundamental class with desired properties. In fact, this weaker requirement amounts to the theory of tangent-obstruction sheaves of \cite{LiTi}.
For many moduli spaces of derived category objects, semi-perfect obstruction theories are easier to construct (cf. \cite{CLs}) and it was shown in \cite{KLp} that a critical virtual manifold always has a semi-perfect obstruction theory and hence a virtual fundamental class.  
The goal of this paper is to generalize and establish the localization theorems in the setting of semi-perfect obstruction theory. 

\medskip

A semi-perfect obstruction theory on a \DM stack $X$ consists of an \'etale cover $\{X\lalp\to X\}$ and perfect obstruction theories $$\{\phi\lalp:E\lalp\to \bbL_{X\lalp}\}$$ of $X\lalp$ satisfying compatibility conditions on the obstruction sheaves and assignments (cf. Definition \ref{2.10}). By constructing a cone cycle $\cC_X$ in the obstruction sheaf $Ob_X$ and applying the (generalized) Gysin map $0^!_{\cF}$ for the coherent sheaf $\cF=Ob_X$, we obtain the virtual fundamental class $[X]\virt$ with respect to the semi-perfect obstruction theory on $X$. 
When there is an action of $T=\CC^*$ and all the local perfect obstruction theories as well as the \'etale cover $X\lalp\to X$ are $T$-equivariant, we prove that the usual torus localization formula
\[ [X]\virt=\imath_* \frac{[F]\virt}{e(N\virt)}\]
holds (cf. Theorem \ref{4.10}) where $F$ is the $T$-fixed locus equipped with the induced semi-perfect obstruction theory. The only additional assumption we need is the existence of a global resolution $[N_0\to N_1]$ of the virtual normal bundle $N\virt$ of $F$ by locally free sheaves as in \cite{CKL}.

Compatibility of perfect obstruction theories $\phi\lalp:E\lalp\to \bbL_{X\lalp}$ requires that the local obstruction sheaves $Ob_{X\lalp}=h^1(E\lalp^\vee)$ glue to a sheaf $Ob_X$ on $X$. If there is a cosection $\sigma:Ob_X\to \sO_X$, we prove (cf. Theorem \ref{3.0}) that the virtual fundamental class $[X]\virt$ of $X$ localizes to the zero locus $X(\sigma)$ of $\sigma$, generalizing the cosection localization theorem in \cite{KLc}.

In \cite{CKL}, the authors combined the torus localization with the cosection localization. They proved that the torus localization formula holds for the cosection localized virtual fundamental class if the cosection $\sigma$ is $T$-invariant. In \S\ref{sec5}, we generalize this combined localization result to the setting of semi-perfect obstruction theory (cf. Theorem \ref{5.6}).

As an application, we consider the Jiang-Thomas theory of virtual signed Euler characteristic in \cite{JiTh}.
For a scheme $X$ equipped with a perfect obstruction theory, 
Jiang and Thomas collected five natural ways to think of virtual signed Euler characteristic via
\begin{enumerate}
\item K-theory,
\item torus localization,
\item Euler characteristic weighted by the Behrend function,
\item cosection localization and
\item ordinary Euler characteristic,
\end{enumerate}
and proved that (1)=(2) while (3)=(4)=(5).
To make sense of (2), (3) and (4), they constructed the dual obstruction cone $N=\Spec_X(\mathrm{Sym}(Ob_X))$ and 
proved that $N$ admits a perfect obstruction theory after assuming that $X$ comes from a quasi-smooth derived scheme. 
We will prove in \S\ref{sec6} that $N$ always admits a semi-perfect obstruction theory without any assumption from derived algebraic geometry.
Applying the torus localization and cosection localization theorems proved in this paper, we can make sense of (2), (3) and (4) and the Jiang-Thomas theory of virtual signed Euler characteristic is established without the technical assumption.

\medskip

Here is the layout of the paper.
In \S\ref{sec3}, we will prove that the cosection localization of \cite{KLc} works for semi-perfect obstruction theory. 
In \S\ref{sec4}, we will show that the torus localization of Graber-Pandharipande \cite{GrPa} also works for semi-perfect obstruction theory. In \S\ref{sec5}, we prove that the torus localization works for the cosection localized virtual fundamental class of a semi-perfect obstruction theory. In \S\ref{sec6}, we apply these results to the dual obstruction cone of Jiang-Thomas and show that their theory works without the technical quasi-smoothness assumption from derived algebraic geometry.

\medskip
All schemes or \DM stacks in this paper are of finite type and defined over the complex number field $\CC$. All Chow groups in this paper have coefficients in the rational number field $\QQ$.

\bigskip

\section{Localization of virtual fundamental cycles and semi-perfect obstruction theory}\label{sec2}

\subsection{Localization of virtual fundamental classes}\label{S2.1}
In this subsection, we recall the torus localization of Graber-Pandharipande \cite{GrPa} and the cosection localization for virtual fundamental class (cf. \cite{KLc}). 

Let $X$ be a \DM stack over $\CC$. The intrinsic normal cone $\bfc_X$ of $X$ is a cone stack 
which satisfies $[C_{U/V}/T_V|_U]=\bfc_X|_U$ if $U\to X$ is \'etale and $U$ is closed in a smooth variety $V$. 
\begin{defi}\cite{BeFa}
A \emph{perfect obstruction theory} on a \DM stack $X$ is a morphism $\phi:E\to \bbL_X$ in the derived category $D(\sO_X)$ of quasi-coherent $\sO_X$-modules such that
\begin{enumerate}
\item $E$ is locally isomorphic to a two-term complex $[E^{-1}\to E^0]$ of locally free sheaves,
\item $h^{-1}(\phi)$ is surjective and $h^0(\phi)$ is an isomorphism.
\end{enumerate}
The perfect obstruction theory is \emph{symmetric} if there exists an isomorphism $\theta:E\to E^\vee[1]$ with $\theta^\vee[1]=\theta$. 
\end{defi}
Here, $\bbL_X$ denotes the truncated cotangent complex $\tau^{\ge -1}L_X$ of $X$.

A perfect obstruction theory $\phi:E\to \bbL_X$ gives us a vector bundle stack $\cE=h^1/h^0(E^\vee)$ which is locally $E_1/E_0$ where $E_i$ is the dual bundle of $E^{-i}$ for $i=0,1$. Then by \cite{BeFa}, the intrinsic normal cone $\bfc_X$
is canonically embedded into the abelian cone stack $h^1/h^0(\bbL_X^\vee)$
which is embedded into $\cE$ by $\phi^\vee$. Then the virtual fundamental class of $X$ is defined as 
$$[X]\virt=0^!_\cE[\bfc_X]$$
where $0^!_\cE$ is the Gysin map for the vector bundle stack $\cE$. 

When $X$ admits an action of $T=\CC^*$ and the perfect obstruction theory $\phi:E\to \bbL_X$ is 
$T$-equivariant, the virtual fundamental class of $X$ is localized to the $T$-fixed locus $F$ in $X$ by the formula
\beq\label{2.1}
[X]\virt=\imath_*\frac{[F]\virt}{e(N\virt)}\in A^T_*(X)\otimes_{\QQ[t]}\QQ[t,t^{-1}]
\eeq
where $\imath:F\to X$ is the inclusion and $N\virt$ is the moving part of $E^\vee|_F$ if $N\virt$ admits
a global resolution $[N_0\to N_1]$ by locally free sheaves $N_0, N_1$ on $F$ (cf. \cite{GrPa, CKL}).  


When the obstruction sheaf $Ob_X=h^1(E^\vee)$ has a cosection $\sigma:Ob_X\to\sO_X$, the virtual fundamental class $[X]\virt$ is localized to the zero locus $X(\sigma)$ of $\sigma$ by \cite{KLc}.
The cosection localized virtual fundamental class is obtained in two steps. Firstly we observe that the intrinsic normal cone $\bfc_X$ has support in $$\cE(\sigma)=\cE|_{X(\sigma)}\cup \ker(\cE|_U\mapright{\sigma} \sO_U)$$ where $U=X-X(\sigma)$. Secondly the Gysin map $0^!_{\cE}:A_*\cE\to A_*X$ is localized to a map
$$0^!_{\cE,\mathrm{loc}}:A_*(\cE(\sigma))\lra A_*(X(\sigma))$$
and the cosection localized virtual cycle is defined as
$$[X]\virt_\loc=0^!_{\cE,\mathrm{loc}}[\bfc_X].$$ 
This cosection localized virtual fundamental class satisfies many expected properties such as deformation invariance and
$$\imath_*[X]\virt_\loc=[X]\virt$$
where $\imath:X(\sigma)\to X$ is the inclusion map.

Combining the localization techniques is often useful. In \cite{CKL}, the authors proved that the torus localization formula \eqref{2.1} holds for the cosection localized virtual cycles, i.e.
\beq\label{2.2}
[X]\virt_\loc=\imath_*\frac{[F]\virt_\loc}{e(N\virt)}\in A^T_*(X(\sigma))\otimes_{\QQ[t]}\QQ[t,t^{-1}]
\eeq
which turned out to be quite useful. 

\subsection{Cycles on sheaf stacks}\label{S2.2}

In this section, we recall the notion of cycles on sheaf stacks \cite{CLs}. Let $X$ be a \DM stack and $\cF$ be a coherent sheaf on $X$. Then the groupoid associating sections $s\in \Gamma(U,f^*\cF)$ to a morphism $f:U\to X$ is a stack, called the \emph{sheaf stack} of $\cF$. 
By abusing the notation, we denote the sheaf stack by the same symbol $\cF$. 

We will consider the intersection theory on the sheaf stack $\cF$. 
\begin{defi} \cite[Definition 3.2]{CLs}
(1) A \emph{reduced cycle} in $\cF$ is a substack $B$ of $\cF$ such that for any \'etale morphism $f:U\to X$ from a scheme $U$ and a surjective homomorphism $\cV\to f^*\cF$ from a vector bundle $\cV$ on $U$, $\tilde{B}=\cV\times_{f^*\cF}f^*B\subset \cV$ is a Zariski closed reduced subscheme. 

(2) Three  reduced cycles $B_1, B_2, B_3$ in $\cF$ satisfy $B_3=B_1\cup B_2$ if 
for any \'etale morphism $f:U\to X$ from a scheme $U$ and a surjective homomorphism $\cV\to f^*\cF$ from a vector bundle $\cV$ on $U$, $\tilde{B}_3=\tilde{B}_1\cup\tilde{B}_2$  where $\tilde{B}_i=\cV\times_{f^*\cF}f^*B_i$. 

(3) A reduced cycle $B$ in $\cF$ is \emph{irreducible} if it is not the union of two nontrivial reduced cycles $B_1$ and $B_2$ with $B\ne B_1$ and $B\ne B_2$. 

(4) A \emph{prime cycle} in $\cF$ is an irreducible reduced cycle in $\cF$.

(5) The $\QQ$-vector space spanned by prime cycles in $\cF$ is denoted by $Z_*(\cF)$. Elements of $Z_*(\cF)$ are called \emph{cycles} on $\cF$. 
\end{defi}

\begin{exam} \cite{Behr}
Let $X$ be a \DM stack equipped with a perfect obstruction theory $E\to \bbL_X$. 
Let $$\bfc_X\subset \cE_X=h^1/h^0(E^\vee)$$ be the intrinsic normal cone of $X$. Then the obstruction sheaf $Ob_X=h^1(E^\vee)$ is the coarse moduli sheaf of $\cE_X$ and let 
$$\cC_X\hookrightarrow Ob_X$$
denote the coarse moduli sheaf of $\bfc_X$ so that we have a cartesian diagram
$$\xymatrix{
\bfc_X\ar@{^(->}[r] \ar[d] &\cE_X\ar[d]\\
\cC_X\ar@{^(->}[r] &Ob_X.
}$$ 
Then for any \'etale open $f:U\to X$ and a surjective homomorphism $\cV\to f^*Ob_X$ from a vector bundle $\cV$ on $U$, the fiber product $$\cV\times_{f^*Ob_X}f^*\cC_X\subset \cV$$
gives us a cycle (called the \emph{obstruction cone}) on the vector bundle $\cV$ (cf. \cite[\S2]{Behr}). Therefore $\cC_X\in Z_*(Ob_X)$. 

It was proved in the proof of \cite[Proposition 2.2]{Behr} that if $V\to Ob_X$ is a surjective homomorphism from a vector bundle $V$ on $X$, then the virtual fundamental class is given by
$$[X]\virt = 0^!_{V}[V\times_{Ob_X}\cC_X].$$
In the notation of \cite{CLs}, the right side is $0^!_{Ob_X}(\cC_X)$ so that we have
$$[X]\virt=0^!_{Ob_X}(\cC_X).$$
\end{exam}

Let $X$ be a \DM stack and $\cF$ be a coherent sheaf on $X$. Let $B\in Z_*(\cF)$ be a prime cycle. Let $\rho:U\to X$ be an \'etale morphism with $\rho^*B=B\times_XU\ne 0\in Z_*(\rho^*\cF)$.  Pick a vector bundle $F$ on $U$ that admits a surjective homomorphism $F\to \rho^*\cF$. Let $\bar B\subset U$ be the image of the cycle $\rho^*B\times_{\rho^*\cF}F$ by the projection $F\to U$. Then the closure $Y$ of $\rho(\bar B)$ in $X$ is a closed \DM substack. 
By Chow's lemma \cite[16.6.1]{LaMo}, 
there are a quasi-projective scheme $S$, a generically \'etale  proper surjective morphism $f:S\to Y$, a locally free sheaf $\cV$ and a surjective homomorphism $\cV\to f^*\cF$. For an open $S'\subset S$ such that $f|_{S'}:S'\to Y$ is \'etale, let $\tilde{B}$ be the closure in $\cV$ of $\cV|_{S'}\times_{\cF}B\subset \cV|_{S'}$. 

\begin{defi}\label{2.50}
The triple $(f:S\to Y, \cV\to f^*\cF, \tilde{B}\subset \cV)$ is called a \emph{proper representative} of $B$. 
\end{defi}

Let $W\in Z_*(\cF)$ be a prime cycle on a coherent sheaf $\cF$ on a \DM stack $X$. 
A rational function $h$ on $W$ is an equivalence class of rational functions $h_f\in \mathbf{k}(\tilde{W})$ on $\tilde{W}$ for proper representatives $(f, \cV, \tilde{W})$ of $W$ where the equivalence relation is defined as follows: Let $(f',\cV',\tilde{W}')$ be another proper representative of $W$ and $h_{f'}\in \mathbf{k}(\tilde{W}')$. Then  $h_f\sim h_{f'}$ if and only if there is a third proper representative $(\hat{f},\hat{\cV}, \hat{W})$ of $W$ that fits into the commutative diagrams
\[
\xymatrix{\hat{S}\ar[r]^{g'} \ar[d]_g\ar[dr]^{\hat{f}}  &S'\ar[d]^{f'}\\
S\ar[r]^f & Y
}\and \xymatrix{
\hat{\cV}\ar@{->>}[r]\ar@{->>}[d] &  {g'}^*\cV'\ar@{->>}[d]\\
g^*\cV\ar@{->>}[r] & g^*f^*\cF=\hat{\cF}={g'}^*{f'}^*\cF ,
}\]
such that the pullbacks of $h_f$ and $h_{f'}$ to $\hat{W}$ by $g$ and $g'$ respectively coincide in $\mathbf{k}(\hat{W})$. 

We say a rational function $h$ on a prime cycle $W$ on $\cF$ is \emph{admissible} if for any proper representative $(f,\cV,\tilde{W})$, the extension $h_f^{\mathrm{nor}}:\tilde{W}^{\mathrm{nor}}\to \PP^1$ of $h_f$ to the normalization $\tilde{W}^{\mathrm{nor}}$  is constant on the fibers of $\tilde{W}^{\mathrm{nor}}\to \cF$. 

For a prime cycle $W$ on $\cF$ and an admissible rational function $h$ on $W$, we write the principal divisor of $h_f$ on $\tilde{W}$ as 
$$\partial (\tilde{W},h_f)=\sum_i n_iD_i$$
for distinct prime divisors $D_i$ on $\tilde{W}$ where $(f,\cV, \tilde{W})$ is a proper representative of $W$ and $h_f$ is a rational function on $\tilde{W}$ representing $h$. Then the image $W_i$ of $D_i$ in $\cF$ is a prime cycle on $\cF$ (cf. \cite[\S3]{CLs}). The boundary of $(W,h)$ is now defined as
$$\partial (W,h)=\sum_i n_i {W}_i \frac{e_i}{\deg f}$$
where $e_i$ is the degree of the morphism $\bar{W}_i\to W_i$ where $\bar{W}_i$ is the image of $D_i$ in $f^*\cF$.

\begin{defi}\label{2.53} \cite{CLs}
Two cycles $B_1$ and $B_2$ on a coherent sheaf $\cF$ on a \DM stack $X$ are \emph{rationally equivalent} if $B_1-B_2$ is a linear combination of  cycles of the form $\partial (W,h)$ where $W$ is a prime cycle on $\cF$ and $h$ is an admissible rational function on $W$. We let 
$A_*(\cF)$ be the $\QQ$-vector space of rational equivalence classes of cycles in $Z_*(\cF)$. 
\end{defi}


With this preparation, the Gysin map for a coherent sheaf $\cF$ is now defined as follows. 
\begin{defi}\label{2.54}
The Gysin map for $\cF$ is defined by
$$0_\cF^![B]=\frac{1}{\deg f} f_*(0_\cV^![\tilde B])\in A_*X$$
where $(f,\cV,\tilde{B})$ is a proper representative of $B$. 
\end{defi}
By \cite[Proposition 3.4]{CLs}, $0^!_{\cF}[B]$ is independent of the choice of a proper representative of $B$.
By \cite[Corollary 3.6]{CLs}, the Gysin map $0^!_\cF$ preserves rational equivalence and hence induces a map $0_\cF^!:A_*(\cF)\lra A_*(X).$

\medskip

Suppose there is a homomorphism $\sigma:\cF\to \sO_X$, called a \emph{cosection}, of a coherent sheaf $\cF$ on a \DM stack $X$. Let $X(\sigma)$ be the zero locus of $\sigma$ whose defining ideal is the image of $\sigma$. Let $U=X-X(\sigma)$ and consider the substack
$$\cF(\sigma)=\cF|_{X(\sigma)}\cup \ker(\sigma:\cF|_U\to \sO_U).$$
A cycle on $\cF(\sigma)$ is a cycle on $\cF$ which is a substack of $\cF(\sigma)$, i.e. a prime cycle $B\in Z_*(\cF)$ belongs to $Z_*(\cF(\sigma))$ if for a proper representative $(f:S\to X,\cV\to f^*\cF,\tilde B)$ of $B$, we have 
$$\tilde B\subset \cV|_{S\times_XX(\sigma)}\cup \ker(\tilde \sigma:\cV\to \sO_S)$$
where $\tilde\sigma:\cV\to \sO_S$ is the composition of the surjective homomorphism $\cV\to f^*\cF$ with $f^*\sigma:f^*\cF\to\sO_S$. We say two cycles $B_1$ and $B_2$ on $\cF(\sigma)$ are rationally equivalent if $B_1-B_2$ is a linear combination of cycles of the form $\partial(W,h)$  where $W$ is a prime cycle on $\cF(\sigma)$ and $h$ is an admissible rational function on $W$. We let $A_*(\cF(\sigma))$ be the $\QQ$-vector space of rational equivalence classes of cycles on $\cF(\sigma)$.

\subsection{Semi-perfect obstruction theory}\label{S2.3}
A semi-perfect obstruction theory is a generalization of a perfect obstruction theory which still gives rise to a virtual fundamental class. 

\begin{defi}\label{2.10} \cite{CLs}
A \emph{semi-perfect obstruction theory} on a \DM stack $X$ consists of an \'etale cover $\{X\lalp\to X\}$ and a perfect obstruction theory $\phi\lalp:E\lalp\to \bbL_{X\lalp}$ for each $\alpha$ such that 
\begin{enumerate}
\item there are isomorphisms $\psi_\albe:h^1(E^\vee\lalp)|_{X_\albe}\to h^1(E^\vee_\beta)|_{X_\albe}$ for pairs of indices $\alpha,\beta$, that glue $\{h^1(E^\vee\lalp)\}$ to a sheaf $Ob_X$ on $X$ and that 
\item the perfect obstruction theories $E\lalp|_{X_\albe}$ and $E_\beta|_{X_\albe}$ give the same obstruction assignment via $\psi_\albe$ for pairs of indices $\alpha,\beta$.
\end{enumerate}
\end{defi}
Here $X_\albe=X\lalp\times_XX_\beta$ as usual.

The second condition in Definition \ref{2.10} means the following. 
\begin{defi}
Let $x\in U$ be a closed point of a \DM stack. An \emph{infinitesimal lifting problem} on $U$ consists of an extension $0\to I\to B\to \bar B\to 0$ of Artin local rings with $I\cdot m_B=0$ and a morphism $\bar g:\Spec \bar B\to U$ that sends the closed point of $\Spec \bar B$ to $x$.

If $\phi:E\to \bbL_U$ is a perfect obstruction theory on $U$, letting $\bar \Delta=\Spec \bar B$ and $\Delta=\Spec B$,  the natural morphism $\bbL_{\bar \Delta}\to \bbL_{\bar \Delta/\Delta}=I[1]$ composed with $\phi$ gives us 
\[ \bar g^*E\mapright{\bar g^*\phi}\bar g^*\bbL_X\mapright{\bar g} \bbL_{\bar \Delta}\lra \bbL_{\bar \Delta/\Delta}=I[1] \]
which is the \emph{obstruction class}
\beq
ob_U(\phi,\bar g,B,\bar B)=(\bar g^*E\to I[1])\in Ext^1(\bar g^*E,I)=I\otimes_\CC h^1(E^\vee)|_x .
\eeq

Two obstruction theories $\phi:E\to \bbL_U$ and $\phi':E'\to \bbL_U$ give the \emph{same obstruction assignment} via an isomorphism $\psi:h^1(E^\vee)\mapright{\cong} h^1({E'}^\vee)$ of obstruction sheaves if 
$$ob_ U(\phi',\bar g, B,\bar B)=\psi \left( ob_U(\phi,\bar g, B, \bar B)\right) \in I\otimes_\CC h^1({E'}^\vee)|_x$$
for any infinitesimal lifting problem $(\bar g, B,\bar B).$
\end{defi}

It was proved in \cite{CLs} that the conditions in Definition \ref{2.10}  guarantee that the images of 
$$\bfc_{X\lalp}\hookrightarrow \cE\lalp\lra h^1(E^\vee\lalp)=Ob_{X\lalp}$$
glue to a cone $\cC_X\subset Ob_X$ by the following. 
\begin{prop} \cite{CLs}
Let $\fN_X=h^1/h^0(\bbL_X^\vee)$ denote the abelian hull of the intrinsic normal cone $\bfc_X$ of $X$. Then the morphisms $$\eta\lalp: \fN_X|_{X\lalp}= h^1/h^0(\bbL_{X\lalp}^\vee)\lra h^1/h^0(E\lalp^\vee)\lra h^1(E\lalp^\vee)=Ob_X|_{X\lalp}$$
glue to a morphism $\eta:\fN_X\to Ob_X$. Moreover, if $A\in Z_*(\fN_X)$, then $\{(\eta\lalp)_*(A|_{X\lalp})\}$ glue to a cycle $\eta_*[A]\in Z_*(Ob_X)$. 
\end{prop}


Applying the Gysin map
$$0^!_{\cF}:A_*(\cF)\lra A_*(X),\quad \cF=Ob_X$$ 
to the image $\cC_X\in Z_*(Ob_X)$ of the intrinsic normal cone $\bfc_X\in Z_*(\fN_X)$, we get the virtual fundamental class $[X]\virt$ of $X$ for the semi-perfect obstruction theory.

We will show below that the torus localization of Graber-Pandharipande \cite{GrPa} (cf. \S\ref{sec4}) and the cosection localization \cite{KLc} (cf. \S\ref{sec3}) as well as their combination (cf. \S\ref{sec5}) work for semi-perfect obstruction theories.

\bigskip

\section{Cosection localization for semi-perfect obstruction theory}\label{sec3}

In this section, we generalize the cosection localization principle in \cite{KLc} to the setting of semi-perfect obstruction theory in \cite{CLs}. 

Let $X$ be a \DM stack over $\CC$ equipped with a semi-perfect obstruction theory
\beq\label{3.1}
(X\lalp\to X,\quad \phi\lalp:E\lalp\to \bbL_{X\lalp} ).
\eeq 
Let $Ob_X$ denote the obstruction sheaf, the gluing of $h^1(E\lalp^\vee)$. Let
\beq\label{3.2}
\sigma:Ob_X\lra \sO_X  
\eeq
be a homomorphism, called a \emph{cosection} of the obstruction sheaf. 
Let $X(\sigma)$ be the closed substack (\emph{zero locus} of $\sigma$) of $X$ defined by the ideal $\sigma(Ob_X)\subset \sO_X$
and let $U=X-X(\sigma)$ be the open substack over which $\sigma$ is surjective. 

The goal of this section is to prove the following generalization of \cite[Theorem 1.1]{KLc}.
\begin{theo}\label{3.0} (Cosection localization for semi-perfect obstruction theory) \\
Let $X$ be a \DM stack equipped with a semi-perfect obstruction theory. Suppose the obstruction sheaf $Ob_X$ admits a cosection $\sigma:Ob_X\to \sO_X$. Let $X(\sigma)$ be the zero locus of $\sigma$ and 
$\imath:X(\sigma)\to X$ denote the inclusion. 
Then there exists a localized virtual fundamental class
\[ [X]\virt_\loc\in A_*(X(\sigma)) \]
satisfying $\imath_*[X]\virt_\loc=[X]\virt\in A_*(X)$.
\end{theo}


%


Since $Ob_X|_{X\lalp}=h^1(E\lalp^\vee)$, we have a canonical morphism of stacks 
\beq\label{3.4}
\cE\lalp=h^1/h^0(E\lalp^\vee)\lra h^1(E\lalp^\vee)=Ob_X|_{X\lalp}.
\eeq
Let $U=X-X(\sigma)$ and $U\lalp=U\times_XX\lalp$. The cosection $\sigma$ in \eqref{3.2} together with
 \eqref{3.4} induces a morphism
\[
\bar\sigma\lalp: \cE\lalp|_{U\lalp}\lra Ob_{U\lalp}\mapright{\sigma} \sO_{U\lalp}.
\]
Then by \cite[Proposition 4.3]{KLc}, $[\mathbf{c}_{X\lalp}]\in Z_*(\cE\lalp(\sigma))$ where 
\[
\cE\lalp(\sigma)=\cE\lalp|_{X\lalp(\sigma)}\cup \ker[\bar\sigma\lalp: \cE\lalp|_{U\lalp}\rightarrow \sO_{U\lalp}],\quad X\lalp(\sigma)=X\lalp\times_XX(\sigma).
\]
Hence the image of $[\mathbf{c}_{X\lalp}]$ by $\cE\lalp\to Ob_{X\lalp}$ lies in $Z_*(Ob_{X\lalp}(\sigma))$ where $$Ob_{X\lalp}(\sigma)=Ob_{X\lalp}|_{X\lalp(\sigma)}\cup \ker[\sigma\lalp:Ob_{U\lalp}\to \sO_{U\lalp}].$$ Since $\cC_X$ is the gluing of the images of $[\mathbf{c}_{X\lalp}]$ and $Ob_X$ is the gluing of $Ob_{X\lalp}$ while $\sigma$ is the gluing of $\sigma\lalp$, 
letting $$Ob_X(\sigma)=Ob_X|_{X(\sigma)}\cup \ker[\sigma_U:Ob_U\twoheadrightarrow \sO_U],$$ 
we find that 
\beq\label{3.8}
[\cC_X]\in A_*(Ob_X(\sigma)).
\eeq

\medskip 

Next, we generalize the cosection localized Gysin map (cf. \cite[Proposition 1.3]{KLc}).
\begin{prop}\label{3.9}
Let $\cF$ be a coherent sheaf on a \DM stack $X$ and let $\sigma:\cF\to \sO_X$ be a nonzero cosection. 
Let $X(\sigma)$ denote the zero locus of $\sigma$ and $U=X-X(\sigma)$ so that $\sigma$ is surjective over $U$. Let 
$\cF(\sigma)=\cF|_{X(\sigma)}\cup \ker[\sigma_U:\cF|_U\twoheadrightarrow \sO_U].$ Then there is a homomorphism
\beq\label{3.10}
0_{\cF,\sigma}^!:A_*(\cF(\sigma))\lra A_*(X(\sigma))
\eeq
which we call the \emph{localized Gysin map}, such that 
\beq\label{3.12}
\imath_*\circ 0^!_{\cF,\sigma}=0_\cF^!\circ \tilde{\imath}_*:A_*\cF(\sigma)\lra A_*(X)
\eeq
where $\imath:X(\sigma)\to X$ and $\tilde{\imath}:\cF(\sigma)\to \cF$ denote the inclusions and $0_\cF^!$ is the Gysin map in Definition \ref{2.54}.  
\end{prop}

\begin{proof}
Let $B\in Z_*(\cF(\sigma))$ be a prime cycle. If $B\in Z_*(\cF|_{X(\sigma)})$, then we let
\beq\label{3.50}
0^!_{\cF,\sigma}[B]:=0^!_{\cF|_{X(\sigma)}}[B]\in A_*(X(\sigma))
\eeq
where $0^!_{\cF|_{X(\sigma)}}$ is the Gysin map of the sheaf $\cF|_{X(\sigma)}$ on $X(\sigma)$
defined in \cite[\S3]{CLs}. From now on, we let the prime cycle $B$ satisfy $B\notin Z_*(\cF|_{X(\sigma)})$. 

\def\tX{X^\#}
\def\tcF{{\cF^\#}}
\def\cI{\mathcal{I}}
\def\trho{\tilde{\rho}}
Let $\pi:\cF\to X$ denote the projection. Let $\rho:{X^\#}\to X$ be the blowup along $X(\sigma)$, i.e.  along the ideal $\cI_\sigma=\sigma(\cF)\subset \sO_X$ so that we have a surjective homomorphism
\beq\label{3.51}
\sigma^\#:\tcF=\rho^*\cF\lra \sO_{{X^\#}}(D)
\eeq
where $D$ is a Cartier divisor on $\widetilde{X}$ with support $|D|$ such that $\sO_{X^\#}(D)=\rho^{-1}\cI_\sigma\cdot\sO_{{X^\#}}$ and $\rho(|D|)\subset X(\sigma)$.  

For a proper representative $(f:S\to X,\cV\twoheadrightarrow f^*\cF,\tilde{B})$, let $S^\#$ be the blowup of $S$ along $S(\sigma)=X(\sigma)\times_XS$, so that we have a commutative diagram
\[ \xymatrix{
S^\#\ar[r]^{f^\#}\ar[d]_{\rho^\#} & X^\#\ar[d]^\rho\\
S\ar[r]_f & X.
}\]
Let $-D_S$ denote the exceptional divisor of $\rho^\#$. Let $\tilde{\cF}=f^*\cF$, $\cV^\#=(\rho^\#)^*\cV$ and $\tilde{\cF}^\#=(\rho\circ f^\#)^*\cF=(f\circ\rho^\#)^*\cF$. Let $\tilde{B}^\#\subset \cV^\#$ be the proper transform of $\tilde{B}\subset \cV$. The cosection $\sigma:\cF\to \sO_X$ and \eqref{3.51} induce a surjective homomorphism 
$$\tilde{\sigma}^\#:{\cV}^\#\twoheadrightarrow \tilde{\cF}^\#\twoheadrightarrow \sO_{S^\#}(D_S).$$ 
We now define 
\beq\label{3.91}
0^!_{\cF,\sigma}[B]:=\frac{1}{\deg f} {f}'_*(\rho^\#)'_*(D_S\cdot 0^!_{\ker\tilde{\sigma}^\#}[\tilde B^\#])\in A_*(X(\sigma)), 
\eeq
where ${f}':S(\sigma)\to X(\sigma)$ is the restriction of $f$ to $S(\sigma)$ and $(\rho^\#)':D_S\to S(\sigma)$ is the restriction of $\rho^\#$ to $D_S$. 
Here $D\cdot$ denotes the intersection with the divisor $D$ (cf. \cite[Chapter 2]{Fulton}).
By the definition of the localized Gysin map for vector bundles in \cite{KLc}, we have
\beq\label{3.92}
0^!_{\cV,\sigma}[\tilde B]=(\rho^\#)'_*(D_S\cdot 0^!_{\ker\tilde{\sigma}^\#}[\tilde B^\#])\in A_*(S(\sigma))\eeq 
so that we have
\beq\label{3.93}
0^!_{\cF,\sigma}[B]=\frac{1}{\deg f} {f}'_*0^!_{\cV,\sigma}[\tilde B]\in A_*(X(\sigma)).\eeq
It was also proved in \cite[\S2]{KLc} that $0^!_{\cV,\sigma}$ preserves rational equivalence. 

By the usual argument used in \cite[\S2]{KLc} or \cite[\S3]{CLs}, it is straightforward to see that $0^!_{\cF,\sigma}[B]$ is independent of the choice of a proper representative of $B$. Indeed, if we have two proper representatives $(f_1,\cV_1,\tilde{B}_1)$ and $(f_2,\cV_2,\tilde{B}_2)$, we can choose a third proper representative $(\hat{f},\hat{\cV},\hat{B})$ dominating the previous two. Then we can compare the cycles defined by the right side of \eqref{3.91} on $\hat{\cV}$. We omit the detail here. 


By extending \eqref{3.50} and \eqref{3.91} linearly, we obtain a homomorphism
\beq\label{3.53}
0^!_{\cF,\sigma}:Z_*(\cF(\sigma))\lra A_*(X(\sigma)).
\eeq
To prove that \eqref{3.53} preserves rational equivalence, 
let $W$ be a prime cycle on $\cF(\sigma)$ and $h$ be an admissible rational function on $W$. When $W\in Z_*(\cF|_{X(\sigma)})$, this fact was proved in \cite{CLs}. So we may assume $W\notin Z_*(\cF|_{X(\sigma)}).$ 
For a proper representative $(f:S\to X, \cV\to f^*\cF,\tilde{W})$ of $W$ and a rational function $h_f$ on $\tilde{W}\subset \cV$ representing $h$, we write the principal divisor of $h_f$ on $\tilde{W}$ as 
$\partial(\tilde{W},h_f)=\sum n_iD_i$ where $D_i$ are distinct prime divisors.
As we reviewed in \S\ref{S2.2} the principal divisor $\partial(W,h)$ of $h$ on $W$ is defined as 
$\partial (W,h)=\sum n_i W_ie_i/{\deg f}$ where $W_i$ is the image of $D_i$ in $\cF$ and $e_i$ is the degree of the morphism $\bar{W}_i\to W_i$ where $\bar{W}_i$ is the image of $D_i$ in $f^*\cF$. 
By \cite[\S3]{CLs}, if $(f_i,\cV_i,\tilde{W}_i)$ is a proper representative of $W_i$, then 
$$\frac1{e_i}f'_*0^!_{\cV,\sigma}D_i=\frac1{\deg f_i}{f'_i}_*0^!_{\cV_i,\sigma}[\tilde W_i]=0^!_{\cF,\sigma}[W_i].$$ Therefore we have
$$0^!_{\cF,\sigma}\partial(W,h)=0^!_{\cF,\sigma}\sum n_iW_i\cdot \frac{e_i}{\deg f}=\frac1{\deg f}\sum n_i\cdot e_i\cdot 0^!_{\cF,\sigma}W_i$$
$$=\frac1{\deg f}\sum n_i f'_*0^!_{\cV,\sigma}D_i=\frac1{\deg f}f'_*0^!_{\cV,\sigma}\sum n_iD_i=\frac1{\deg f}f'_*0^!_{\cV,\sigma}\partial(\tilde W,h_f).$$
Since $0^!_{\cV,\sigma}(\partial(\tilde{W},h_f))=0$ by \cite[\S2]{KLc}, we have 
$$0^!_{\cF,\sigma}\partial(W,h)=0$$
as desired.

Finally we prove \eqref{3.12}. For a prime cycle $B\in Z_*(\cF|_{X(\sigma)})$ and a proper representative $(f:S\to X, \cV\to f^*\cF, \tilde{B})$, we have a commutative square
$$\xymatrix{
S(\sigma)\ar[r]^{f'}\ar[d]_{\jmath} &X(\sigma)\ar[d]^{\imath}\\
S\ar[r]_f & X.
}$$
Since $\jmath_*0^!_{\cV,\sigma}[\tilde B]=0^!_\cV[\tilde B]$ by \cite[\S2]{KLc}, \eqref{3.12} follows from 
\beq\label{3.68}
\imath_*0^!_{\cF,\sigma}[B]=\frac{1}{\deg f} \imath_*f'_*0^!_{\cV,\sigma}[\tilde B]=\frac{1}{\deg f}f_*\jmath_*0^!_{\cV,\sigma}[\tilde B]=\frac{1}{\deg f}f_*0^!_{\cV}[\tilde B]=0^!_\cF[B].
\eeq
\end{proof}

Note that our definition \eqref{3.93} of the localized Gysin map $0^!_{\cF,\sigma}$ is the same as Definition \ref{2.54} with the ordinary Gysin map $0^!_{\cV}$ replaced by the localized Gysin map $0^!_{\cV,\mathrm{loc}}=0^!_{\cV,\sigma}$ defined in \cite[\S2]{KLc}.

\begin{proof}[Proof of Theorem \ref{3.0}]
By \eqref{3.8}, we have the virtual normal cone $[\cC_X]\in A_*(Ob_X(\sigma))$.
Using $0^!_{\cF,\sigma}$ in \eqref{3.10} with $\cF=Ob_X$, we define
\[
[X]\virt_\loc=0^!_{Ob_X,\sigma}[\cC_X]\in A_*(X(\sigma)) .
\] 
Since $[X]\virt=0^!_{Ob_X}[\cC_X]$ by definition, 
$$\imath_*[X]\virt_\loc=\imath_*0^!_{Ob_X,\sigma}[\cC_X]=0^!_{Ob_X}[\cC_X]=[X]\virt \in A_*(X)$$
by \eqref{3.12} \end{proof}

As in the case of ordinary virtual fundamental class, the localized virtual cycle $[X]\virt_\loc$ for a semi-perfect obstruction theory remains constant under deformation. Let $X$ be a \DM stack over $\CC$. Let $\{X\lalp\to X\}$ be an \'etale cover and let $\{\phi\lalp: E\lalp\to \bbL_{X\lalp} \}$ be a semi-perfect obstruction theory for $X$. 
Let $\fX$ be a \DM stack together with a morphism $\pi:\fX\to T$ to a pointed smooth curve $0\in T$. 

Suppose $X=\fX\times_T\{0\}$ and there is a semi-perfect obstruction theory
\beq\label{3.22}
\{\fX\lalp\to \fX\}, \quad \{\psi\lalp:F\lalp\to \bbL_{\fX\lalp}\},\quad X\lalp=\fX\lalp\times_T\{0\} 
\eeq
for $\fX$ together with commutative diagrams
\beq\label{3.23}
\xymatrix{
F\lalp|_{X\lalp}\ar[r]^{g\lalp} \ar[d] & E\lalp \ar[r]\ar[d] &\sO_{X\lalp}[1]\ar[d]^=\ar[r]&\\
\bbL_{\fX\lalp}|_{X\lalp}\ar[r]^{h\lalp} &\bbL_{X\lalp} \ar[r] & \bbL_{X\lalp/\fX\lalp}\ar[r] &
}\eeq
of distinguished triangles. We further assume that the homomorphisms
$$h^1(g\lalp^\vee):h^1(E\lalp^\vee)\lra h^1(F\lalp|_{X\lalp}^\vee)=h^1(F\lalp^\vee)|_{X\lalp}$$ 
glue to a homomorphism
\beq\label{3.24}
Ob_X\lra Ob_\fX|_X.
\eeq
Let us suppose that there is a homomorphism
\beq\label{3.25}
\tilde{\sigma}:Ob_\fX\lra  \sO_\fX
\eeq
and let 
$\sigma:Ob_X\to Ob_\fX|_X\mapright{\tilde \sigma|_X} \sO_X$ denote the composition of \eqref{3.25} and \eqref{3.24}, restricted to $X$. Let $\fX(\tilde \sigma)$ be the zero locus of $\tilde{\sigma}$ and $X(\sigma)$ be the zero locus of $\sigma$
so that we have a fiber square
\[\xymatrix{
X(\sigma)\ar[r]\ar[d] &\fX(\tilde\sigma)\ar[d]\\
\{0\}\ar[r]^\tau & T.
}\]
Let 
\beq\label{3.26}
\tau^!: A_*(\fX(\tilde\sigma))\lra A_*(X(\sigma))
\eeq
denote the Gysin map. 

\begin{theo}\label{3.30} $\tau^![\fX]\virt_\loc=[X]\virt_\loc.$
\end{theo}
By pushing the cycles to obstruction sheaves via the natural transformation $h^1/h^0\to h^1$, the proof is a straightforward adaptation of the proof of \cite[Theorem 5.2]{KLc} and so we omit it. See also the proof of \cite[Proposition 3.8]{CLs}.


\bigskip

\section{Torus localization for semi-perfect obstruction theory}\label{sec4}

In this section, we generalize the torus localization formula in \cite{GrPa} to the setting of semi-perfect obstruction theory. See \cite{CKL} for the torus localization formula for cosection localized virtual fundamental class. 

Let $X$ be a  \DM stack acted on by a torus $T=\CC^*$. Let $F$ be the $T$-fixed locus, locally defined by $\Spec A/A^{mv}$ on an equivariant \'etale $\Spec A\to X$, where $A^{mv}$ denotes the ideal generated by weight spaces corresponding to nontrivial $T$-weights. Let 
\[\imath:F\hookrightarrow X\]
denote the inclusion map. 

\begin{defi}\label{4.1}
A \emph{$T$-equivariant semi-perfect obstruction theory} on $X$ consists of 
\begin{enumerate}
\item a $T$-equivariant \'etale open cover $\{X\lalp\to X\}$ of $X$;
\item an object $E\lalp\in D([X\lalp/T])$ and a morphism 
$$\phi\lalp:E\lalp\to \bbL_{X\lalp}\text{  in } D([X\lalp/T])$$
\end{enumerate}
which is a perfect obstruction theory on $X\lalp$. 
\end{defi}
Here $D([X\lalp/T])$ denotes the derived category of $T$-equivariant quasi-coherent sheaves on $X\lalp$. 

The $T$-fixed locus $F$ admits an induced semi-perfect obstruction theory. Indeed, the $T$-equivariant perfect complex $E\lalp\in D^b([X\lalp/T])$ restricted to $F\lalp=F\times_XX\lalp$ decomposes into a direct sum
\beq\label{4.2}
E\lalp|_{F\lalp}=E\lalp|_{F\lalp}^{fix}\oplus E\lalp|_{F\lalp}^{mv}
\eeq
of the $T$-fixed part and moving part. Moreover $\phi\lalp:E\lalp\to \bbL_{X\lalp}$ also decomposes into the sum of 
\beq\label{4.3}
\phi\lalp^{fix}:E\lalp|_{F\lalp}^{fix}\to \bbL_{X\lalp}|_{F\lalp}^{fix} \and  \phi\lalp^{mv}:E\lalp|_{F\lalp}^{mv}\to \bbL_{X\lalp}|_{F\lalp}^{mv}.
\eeq
Since $F\lalp$ is $T$-fixed, the morphism $\bbL_{X\lalp}|_{F\lalp}\to \bbL_{F\lalp}$ factors through
\beq\label{4.4}
\bbL_{X\lalp}|_{F\lalp}^{fix}\lra \bbL_{F\lalp}.
\eeq
Let
\beq\label{4.7}
E\lalp^F:=E\lalp|_{F\lalp}^{fix},\quad N\lalp\virt:=(E\lalp|_{F\lalp}^{mv})^\vee.
\eeq
\begin{lemm}\label{4.5}
The composition of \eqref{4.3} and \eqref{4.4} gives us a perfect obstruction theory
\beq\label{4.6}
\phi\lalp^F:E\lalp^F\lra \bbL_{F\lalp}.
\eeq
These perfect obstruction theories $\{\phi\lalp^F:E\lalp^F\to\bbL_{F\lalp}\}$ form a semi-perfect obstruction theory on $F$, with respect to the \'etale cover $\{F\lalp\to F\}$.
\end{lemm}
\begin{proof}
The first statement follows from \cite[Lemma 3.3]{CKL}. 
For the second, observe that we have a decomposition
\[
Ob_{X\lalp}|_{F\lalp}=h^1(E\lalp^\vee)|_{F\lalp}=h^1(E\lalp|_{F\lalp}^\vee)=h^1({E^F\lalp}^\vee)\oplus h^1(N\lalp\virt)
\]
into the direct sum of the $T$-fixed and moving parts. 
Since $\{Ob_{X\lalp}\}$ glue to $Ob_X$ by assumption, $\{Ob_{X\lalp}|_{F\lalp}\}$ 
glue to $Ob_X|_F$ and hence $$\{Ob_{F\lalp}:=h^1((E^F\lalp)^\vee)\}$$ glue to 
\beq\label{4.15}
Ob_F:=Ob_X|_F^{fix}.
\eeq

By composing with the inclusion $F\lalp\hookrightarrow X\lalp$, 
an infinitesimal lifting problem 
\[
\bar{g} :\Spec \bar B\lra F\lalp,\quad \bar B=B/I,\quad I\cdot m_B=0
\]
for $F\lalp$ with obstruction class
\[ ob_{F\lalp}(\phi^F\lalp,\bar g, B,\bar B)\in I\otimes_\CC Ob_{X\lalp}|_x^{fix}=I\otimes_\CC Ob_{F\lalp}|_x  \]
gives us an infinitesimal lifting problem 
\[
\bar g':\Spec \bar B\lra F\lalp\mapright{\imath} X\lalp
\]
for $X\lalp$ with obstruction class
\[ ob_{X\lalp}(\phi\lalp,\bar g', B,\bar B)\in I\otimes_\CC Ob_{X\lalp}|_x \]
where $x\in F$ is the image of the closed point in $\Spec \bar B$ by $\bar g$. 
Since $\bar g$ is $T$-invariant,
\[
ob_{X\lalp}(\phi\lalp,\bar g',B,\bar B)=ob_{F\lalp}(\phi^F\lalp,\bar g,B,\bar B)\in I\otimes_\CC Ob_{F\lalp}|_x\subset I\otimes_\CC Ob_{X\lalp}|_x.
\]
Since the perfect obstruction theories $E\lalp|_{X_\albe}$ and $E_\beta|_{X_\albe}$ give the same obstruction assignment for $X_\albe=X\lalp\times_XX_\beta$, we find that
the same holds for $E\lalp|^{fix}_{F_\albe}$ and $E_\beta|^{fix}_{F_\albe}$. This proves the lemma.
\end{proof}

Recall that in \cite{GrPa}, the authors proved the torus localization formula assuming the existence of \begin{enumerate}
\item a global $T$-equivariant embedding of $X$ into a smooth \DM stack and 
\item a global resolution $[E^{-1}\to E^0]$ of the perfect obstruction theory $E\to\bbL_X$ by locally free sheaves $E^{-1}$ and $E^0$. 
\end{enumerate}
In \cite[\S3]{CKL}, (1) was completely removed and (2) was replaced by a weaker assumption that the virtual normal bundle $N\virt=(E|_F^{mv})^\vee$ admits a global 2-term resolution $[N_0\to N_1]$ by locally free sheaves $N_0, N_1$ over $F$. 

For the torus localization formula in the setting of semi-perfect obstruction theory, 
we assume the following. 

\begin{assu}\label{4.8}
There is a homomorphism
\beq\label{4.9}
N\virt=[N_0\to N_1]
\eeq
of locally free sheaves on $F$ whose cokernel $h^1(N\virt)$ is isomorphic to $Ob_X|_F^{mv}$ such that there is an isomorphism
\beq\label{4.22}
N\virt|_{F\lalp}\cong N\lalp\virt = (E\lalp|_{F\lalp}^{mv})^\vee,\quad \forall \alpha 
\eeq
in the derived category $D([F\lalp/T])$ whose $h^1$ is the restriction of the isomorphism $h^1(N\virt)\cong Ob_X|_F^{mv}$ to $F\lalp$. 
\end{assu}

\begin{defi}\label{4.12}
Under Assumption \ref{4.8}, the Euler class of $N\virt$ is 
\[ e(N\virt)=\frac{e(N_0)}{e(N_1)}\in A_*(F)\otimes_\QQ \QQ[t,t^{-1}] \]
where $e(-)$ denotes the equivariant top Chern class and $t$ is the generator of the equivariant Chow ring $A^*_T(pt)$. 
\end{defi}

\begin{theo}\label{4.10}
Under Assumption \ref{4.8}, we have
\beq\label{4.11}
[X]\virt=\imath_* \frac{[F]\virt}{e(N\virt)}\in A^T_*(X)\otimes_{\QQ[t]}\QQ[t,t^{-1}]
\eeq
where $[X]\virt$ and $[F]\virt$ are the virtual fundamental classes of $X$ and $F$ with respect to the semi-perfect obstruction theories $\{\phi\lalp:E\lalp\to \bbL_{X\lalp}\}$ and $\{\phi\lalp^F:E\lalp^F\to \bbL_{F\lalp}\}$ using the notation above.
\end{theo}

To prove Theorem \ref{4.10}, we introduce another semi-perfect obstruction theory for $F$, whose obstruction sheaf is enlarged by $N_1$ so that the inclusion $\imath:F\lalp\to X\lalp$ is virtually smooth. Let
\beq\label{4.13}
\widetilde{E}^F\lalp=E^F\lalp\oplus N_1|_{F\lalp}^\vee[1]\and 
\tilde{\phi}^F\lalp:\widetilde{E}^F\lalp\lra E^F\lalp\mapright{\phi^F\lalp} \bbL_{F\lalp}
\eeq
where the first arrow $\widetilde{E}^F\lalp\to E^F\lalp$ is the projection to the direct summand $E^F\lalp$ defined in \eqref{4.7}. Since $\phi^F\lalp:E^F\lalp\to \bbL_{F\lalp}$ is a perfect obstruction theory by Lemma \ref{4.5}, $\tilde{\phi}\lalp^F$ is also a perfect obstruction theory for $F\lalp$. Since the obstruction sheaves
$\{h^1((E^F\lalp)^\vee)\}$ glue to $Ob_F=Ob_X|_F^{fix}$, the obstruction sheaves 
\[ h^1((\widetilde{E}^F\lalp)^\vee)= h^1((E^F\lalp)^\vee)\oplus h^1(N_1|_{F\lalp}[-1])  \]
glue to 
\[ \widetilde{Ob}_F=Ob_F\oplus N_1 .\]
As $\tilde{\phi}\lalp^F$ factors through $\phi^F\lalp$ in \eqref{4.13}, $\tilde{\phi}^F\lalp|_{F_\albe}$ and $\tilde{\phi}^F_\beta|_{F_\albe}$ give the same obstruction assignment where $F_\albe=F\lalp\times_FF_\beta=X_\albe\times_XF$. Hence we proved the following.

\begin{lemm}\label{4.14}
\eqref{4.13} defines a semi-perfect obstruction theory on $F$ with obstruction sheaf $\widetilde{Ob}_F=Ob_F\oplus N_1$ where $Ob_F=Ob_X|_F^{fix}$ is the obstruction sheaf of $F$ with respect to the semi-perfect obstruction theory in Lemma \ref{4.5}.
\end{lemm}

Note that the new semi-perfect obstruction theory $\{(\widetilde{E}^F\lalp,\tilde{\phi}\lalp^F)\}$
fits into the commutative diagram
\beq\label{4.16} \xymatrix{
E\lalp|_{F\lalp}\ar[r]\ar[d] & \widetilde{E}^F\lalp\ar[d]\ar[r] & N_0^\vee|_{F\lalp}[1]\ar[d]\ar[r]&\\
\bbL_{X\lalp}|_{F\lalp}\ar[r] & \bbL_{F\lalp}\ar[r] & \bbL_{F\lalp/X\lalp}[1]\ar[r]&.
}\eeq
Since $F\hookrightarrow X$ is an embedding, the normal cone $\bfc_{F/X}=C_{F/X}$ embeds into the normal sheaf $\mathfrak{N}_{F/X}$ which is the kernel of \eqref{4.9}. Hence we have the embedding
\beq\label{4.17}
\bfc_{F/X}\hookrightarrow N_0.
\eeq
This gives us the pullback map 
\beq\label{4.18}
\imath^!:A_*(X)\lra A_*(F)
\eeq
defined by
\beq\label{4.19}
\imath^![B]=0^!_{N_0}[\bfc_{B\times_XF/B}]
\eeq
for any prime cycle $B$ in $X$ via the embedding
\[ 
\bfc_{B\times_XF/B}\subset \bfc_{F/X} \subset N_0.
\]

\begin{prop}\label{4.20}
$\imath^![X]\virt=e(N_1)\cap [F]\virt.$
\end{prop}
\begin{proof}
Since $[X]\virt=0^!_{Ob_X}[\cC_X]$ by definition, we have $\imath^![X]\virt=\imath^!0^!_{Ob_X}[\cC_X].$
Recall that $\cC_X$ is 
the substack of $Ob_X$ whose pullback $$\cC_X|_{X\lalp}\in A_*(Ob_X|_{X\lalp})=A_*(h^1(E\lalp^\vee))$$ to $X\lalp$
is the image of the intrinsic normal cone $\bfc_{X\lalp}\in A_*(h^1/h^0(E\lalp^\vee))$. 
Since $F\subset X\subset \cC_X$, the normal cone $\bfc_{F/\cC_X}\in A_*(Ob_X|_F\oplus N_0)$ is defined as a substack whose restriction
$\bfc_{F/\cC_X}|_{F\lalp}\in A_*(Ob_X|_{F\lalp}\oplus N_0|_{F\lalp})$ to $F\lalp=F\times_XX\lalp$ is the image of $\bfc_{F\lalp/\bfc_{X\lalp}}\in A_*(h^1/h^0(E^\vee)|_{F\lalp}\oplus N_0|_{F\lalp}).$ 

Our proof of Proposition \ref{4.20} will follow from the two lemmas below.

\begin{lemm}\label{4.91}
$\imath^![X]\virt=0^!_{Ob_X|_F\oplus N_0}[\bfc_{F/\cC_X}]$.
\end{lemm}
\begin{proof}
Let $\cC_X=\sum n_iB_i$ where $B_i$ are distinct prime cycles in $Ob_X$. 
Let $(f_i:S_i\to X, \cV_i\to f_i^*Ob_X, \tilde{B}_i)$ be a proper representative of $B_i$ and let $F_{S_i}=F\times_XS_i$ so that we have the commutative diagram
\[\xymatrix{
F\ar[r]^\imath & X & Ob_X\ar[l]_\pi \\
F_{S_i}\ar[u]^{g_i}\ar[r]_{\imath_i} & S_i\ar[u]_{f_i} & f_i^*Ob_X\ar[u]\ar[l]^{\pi_{S_i}}& \cV_i\ar@{->>}[l]
}\]

Let $0^!_{\cV_i}[\tilde B_i]=\tilde{T}_i\in A_*(S_i)$ and $T_i=\frac{1}{\deg f_i}{f_i}_*\tilde{T}_i\in A_*(X)$.
Let $\tilde{\pi}_i:\cV_i\to S_i$ denote the bundle projection. Then $[\tilde B_i]=\tilde{\pi}_i^*[\tilde T_i]\in A_*(\cV_i)$ 
and $$[X]\virt=0^!_{Ob_X}[\cC_X]=\sum_i\frac{n_i}{\deg f_i}{f_i}_*0^!_{\cV_i}[\tilde B_i]=\sum_i n_i T_i.$$ Hence the left side of Lemma \ref{4.91} is 
\beq\label{4.97} \imath^![X]\virt=\imath^! \sum_i n_i T_i=\sum_i n_i 0_{N_0}^![\bfc_{F\times_XT_i/T_i}].\eeq
By the proof of \cite[Proposition 1.9]{Fulton}, the rational equivalence $\tilde B_i\sim {\tilde \pi}^*_i[\tilde T_i]$ is given by admissible rational functions and hence we have 
$[B_i]=[\pi^{-1}(T_i)] \in A_*(Ob_X)$
where $\pi:Ob_X\to X$ is the projection. Since $F\subset X$, we have 
$$\bfc_{F\times_X\pi^{-1}(T_i)/\pi^{-1}(T_i)}=\bfc_{F\times_XT_i/T_i}\times_FOb_X|_F$$
and the right side of Lemma \ref{4.91} is 
\beq\label{4.98} 0^!_{N_0\oplus Ob_X|_F}[\bfc_{F/\cC_X}]=\sum_in_i0^!_{N_0\oplus Ob_X|_F}[\bfc_{F\times_XB_i/B_i}]\eeq
$$ =\sum_in_i0^!_{N_0\oplus Ob_X|_F}[\bfc_{F\times_XT_i/T_i}\times_FOb_X|_F]
=\sum_in_i0^!_{N_0}[\bfc_{F\times_XT_i/T_i}].$$
%
%
The lemma follows from \eqref{4.97} and \eqref{4.98}.
\end{proof}

\begin{lemm}\label{4.92}
There exist a coherent sheaf $\Xi$ on $F\times \PP^1$ such that 
\beq\label{4.28}
\Xi|_{F\times \{s\}} = \left\{  \begin{matrix}
\widetilde{Ob}_F=Ob_F \oplus N_1 & s\ne 0\\ 
Ob_X|_F\oplus N_0 & s=0 .
\end{matrix}\right.
\eeq
and a rational equivalence of cycles 
\[[\bfc_{F/\cC_X}]\sim [\cC_F]\in A_*(\Xi).\]
\end{lemm}
\begin{proof}
Consider the double deformation spaces
\[  M\lalp=M^\circ_{F\lalp\times\PP^1/M^\circ_{X\lalp}}\lra \PP^1\times\PP^1 \]
from \cite{KKP} such that 
\[
M\lalp|_{\{t\}\times\PP^1} =\left\{ \begin{matrix} 
M^\circ_{X\lalp} & t\ne 0\\ 
\bfc_{F\lalp\times\PP^1/M^\circ_{X\lalp}} & t=0
\end{matrix}\right.
\]
where $M^\circ_{X\lalp}\to \PP^1$ is the deformation space such that
\[
M^\circ_{X\lalp}|_s=\left\{ \begin{matrix} 
\mathrm{pt} & s\ne 0\\ 
\bfc_{X\lalp} & s=0 .
\end{matrix}\right.
\]
Here $\bfc_{X\lalp}$ is the intrinsic normal cone of $X\lalp$. Note that 
\[
\bfc_{F\lalp\times\PP^1/M^\circ_{X\lalp}} |_s =
\left\{ \begin{matrix} 
\bfc_{F\lalp/\bfc_{X\lalp}} & s=0\\ 
\bfc_{F\lalp} & s\ne 0 .
\end{matrix}\right.
\]

From \eqref{4.2}, \eqref{4.7} and \eqref{4.13}, we have a morphism
\beq\label{4.23}
\lambda:E\lalp|_{F\lalp}\lra \widetilde{E}\lalp^F
\eeq
whose cone is the locally free sheaf $N_0|_{F\lalp}^\vee$. 
Let $\xi\lalp$ be the cone of the morphism
\beq\label{4.24}
(x_0\cdot\mathrm{id},x_1\cdot \lambda): p_1^*E\lalp|_{F\lalp}\otimes p_2^*\sO_{\PP^1}(-1)\lra p_1^*E\lalp|_{F\lalp}\oplus p_1^*\widetilde{E}^F\lalp
\eeq
where $x_0, x_1$ are the homogeneous coordinates of $\PP^1$ and 
$p_1, p_2$ are the projections from $F\times \PP^1$ to $F$ and $\PP^1$ respectively. 

Since $\bfc_{F\lalp\times\PP^1/M^\circ_{X\lalp}}\subset  h^1/h^0(\xi^\vee\lalp)$, 
we have a rational equivalence
\[ [\bfc_{F\lalp/\bfc_{X\lalp}}]\sim [\bfc_{F\lalp}] \]
in $A_*(h^1/h^0(\xi^\vee\lalp))$ and hence an equivalence
\beq\label{4.26}
[\bfc_{F\lalp/\cC_{X\lalp}}]\sim [\cC_{F\lalp}]
\eeq
in $A_*(h^1(\xi^\vee\lalp))$ where $[\cC_{F\lalp}]$ is the image of $[\bfc_{F\lalp}]$ by the canonical morphism 
$h^1/h^0(\xi^\vee\lalp)\to h^1(\xi^\vee\lalp)$.

Note that $\{\cC_{F\lalp}\}$, $\{\cC_{X\lalp}\}$ and $\{\bfc_{F\lalp/\cC_{X\lalp}}\}$ all glue to 
\[ [\cC_F]\in A_*(Ob_F),\quad [\cC_X]\in A_*(Ob_X),\and [\bfc_{F/\cC_X}] \]
respectively. The canonical equivalences \eqref{4.26} glue to an equivalence
\beq\label{4.29}
[\bfc_{F/\cC_X}]\sim [\cC_F] \in A_*(\Xi). 
\eeq

To prove \eqref{4.28}, we take a more careful look at \eqref{4.24}. To simplify the notation, we will often drop $p_1^*$, $p_2^*$, $\cdot|_{F\lalp}$ when the meaning is clear from the context. For example, $E\lalp|_{F\lalp}(-1)$ means $p_1^*E\lalp|_{F\lalp}\otimes p_2^*\sO_{\PP^1}(-1)$  and $N_1^\vee$ often means $N_1|_{F\lalp}^\vee$. 
By \eqref{4.2}, \eqref{4.7}, \eqref{4.13} and \eqref{4.22}, we have 
\[ p_1^*E\lalp|_{F\lalp}\otimes p_2^*\sO_{\PP^1}(-1)=[N_1^\vee\to N_0^\vee](-1)\oplus E^F\lalp(-1),\]
\[ p_1^*E\lalp|_{F\lalp}\oplus p_1^*\widetilde{E}\lalp^F =[N_1^\vee\to N_0^\vee]\oplus [N_1^\vee\to 0] \oplus E\lalp^F \oplus E^F\lalp  .\]
The cone $\xi\lalp^{fix}$ of the fixed part
\[ E^F\lalp(-1)\mapright{(x_0,x_1)} E^F\lalp\oplus E^F\lalp   \]
is certainly $E^F\lalp(1)$ and $\{h^1(E^F\lalp(1)^\vee)\}$ glue to the sheaf
$ p_1^*Ob_F\otimes p_2^*\sO_{\PP^1}(1)$ on $F\times \PP^1$. 
By direct computation, the cone $\xi^{mv}\lalp$ of the moving part has
\beq\label{4.30}
h^1((\xi^{mv}\lalp)^\vee)=h^1(N_0\mapright{(d,0,-x_0)}N_1\oplus N_1\oplus N_0(1)\mapright{(-x_0,-x_1,-d)}N_1(1))
\eeq
where $d$ is the differential of $N\virt=[N_0\mapright{d}N_1]$. 
Since $N_0, N_1, d, x_0,x_1$ are globally defined on $F\times \PP^1$, 
we find that $\{h^1(\xi\lalp)\}$ glue to a sheaf $\Xi$ on $F\times \PP^1$. 
Finally \eqref{4.28} follows from \eqref{4.30}.
\end{proof}

Now we can complete our proof of Proposition \ref{4.20}. 
By Lemma \ref{4.92}, 
\beq\label{4.93} 0^!_\Xi[\cC_F]=0^!_\Xi[\bfc_{F/\cC_X}].\eeq
Since $\cC_F\subset \Xi|_{F\times\{\infty\}}=Ob_F\oplus N_1$, we have 
\beq\label{4.94}
0^!_\Xi[\cC_F]=0^!_{Ob_F\oplus N_1}[\cC_F]=e(N_1)\cap [F]\virt.\eeq
Likewise, since $\bfc_{F/\cC_X}\subset \Xi|_{F\times\{0\}}=Ob_X|_F\oplus N_0$, by Lemma \ref{4.91}, we have 
\beq\label{4.95}
0^!_\Xi[\bfc_{F/\cC_X}]=0^!_{Ob_X|_F\oplus N_0}[\bfc_{F/\cC_X}]=\imath^![X]\virt.\eeq
The proposition follows from \eqref{4.93}, \eqref{4.94} and \eqref{4.95}.
\end{proof}

Now  we can prove Theorem \ref{4.10}.
\begin{proof}[Proof of Theorem \ref{4.10}]
By \cite[Theorem 6.3.5]{Kres}, the inclusion $\imath:F\to X$ induces an isomorphism
\[
\imath_*:A^T_*(F)\otimes_{\QQ[t]}\QQ[t,t^{-1}]\lra A^T_*(X)\otimes_{\QQ[t]}\QQ[t,t^{-1}] .
\]
Hence we have a class $\gamma\in A_*^T(F)\otimes_{\QQ[t]}\QQ[t,t^{-1}]$ such that 
$\imath_*\gamma=[X]\virt$. 

By \eqref{4.19}, $\imath^![B]=0^!_{N_0}(\bfc_{B\times_XF/B})=0^!_{N_0}[B]=e(N_0) \cap [B]$ if $B\subset F$. 
Hence we have
\[ \imath^![X]\virt =\imath^!\imath_*\gamma=e(N_0) \cap \gamma. \]
By Proposition \ref{4.20}, we obtain
\[ \gamma\cap e(N_0)=e(N_1)\cap [F]\virt .\]
Since $e(N\virt)=e(N_0)/e(N_1)$, 
\[ [X]\virt=\imath_*\gamma =\imath_* \frac{[F]\virt}{e(N\virt)}\]
as desired.
\end{proof}


\bigskip

\section{Torus localization of cosection localized virtual cycles for semi-perfect obstruction theory}\label{sec5}

In this section, we generalize the torus localization formula of cosection localized virtual cycle 
(cf. \cite[\S3]{CKL}) to the setting of semi-perfect obstruction theory.

Let $X$ be a \DM stack equipped with 
\begin{enumerate} 
\item an action of $T=\CC^*$,
\item a semi-perfect obstruction theory
\beq\label{5.1} (X\lalp\to X, \phi\lalp:E\lalp\to \bbL_{X\lalp}) ,\quad \text{and}\eeq
\item  a $T$-invariant cosection 
\beq\label{5.2} \sigma:Ob_X\lra \sO_X\eeq
of the obstruction sheaf $Ob_X$ of the semi-perfect obstruction theory.
\end{enumerate}
Let $X(\sigma)$ denote the zero locus of $\sigma$, that is the closed substack defined by the image of $\sigma$.
By Lemma \ref{4.5}, we have an induced semi-perfect obstruction theory 
\beq\label{5.3} (F\lalp\to F, \phi^F\lalp:E^F\lalp\to \bbL_{F\lalp}) \eeq
on $F$. Since $\sigma$ is $T$-invariant, the restriction $\sigma|_F$ of $\sigma$ to $F$ factors through
a homomorphism 
\beq\label{5.4} \sigma_F:Ob_F\lra \sO_F \eeq
whose zero locus is $F(\sigma)=F\times_XX(\sigma)$. 

By Theorem \ref{3.0}, we have cosection localized virtual fundamental classes
\beq\label{5.5}
[X]\virt_\loc\in A_*(X(\sigma)),\quad [F]\virt_\loc\in A_*(F(\sigma)).
\eeq
The goal of this section is to prove the following.
\begin{theo}\label{5.6}
Let $\imath:F(\sigma)\to X(\sigma)$ denote the inclusion. Suppose Assumption \ref{4.8} holds. Then we have
\[
[X]\virt_\loc=\imath_*\frac{[F]\virt_\loc}{e(N\virt)} \in A^T_*(X(\sigma))\otimes_{\QQ[t]}\QQ[t,t^{-1}].
\]
\end{theo}

As in \S\ref{sec4}, Theorem \ref{5.6} follows from the following lemma. 
\begin{lemm}\label{5.7}
$\imath^![X]\virt_\loc= e(N_1)\cap [F]\virt_\loc.$
\end{lemm}
\begin{proof}
The proof is identical to that of Proposition \ref{4.20}, once we show that the rational equivalence in Lemma \ref{4.92} lies in 
$$\Xi(\sigma)=\Xi|_{F(\sigma)\times \PP^1}\cup \ker\{\sigma:\Xi\to \sO\}.$$
Its proof is identical to that of \cite[Lemma 2.5]{CKL}. 
\end{proof}

\begin{proof}[Proof of Theorem \ref{5.6}]
The proof of Theorem \ref{4.10} also proves Theorem \ref{5.6} if we replace $F$, $X$, $[F]\virt$ and $[X]\virt$ by $F(\sigma)$, $X(\sigma)$, $[F]\virt_\loc$ and $[X]\virt_\loc$ respectively. 
\end{proof}

\bigskip

\section{Dual obstruction cone and virtual signed Euler characteristic}\label{sec6}

In this section,  we recall the dual obstruction cone $N$ of a \DM stack $X$ with a perfect obstruction theory by Jiang-Thomas \cite{JiTh} and prove that it admits a semi-perfect obstruction theory (Theorem \ref{6.7}). 
The cosection localization (Theorem \ref{3.0}) and the torus localization (Theorem \ref{4.10}) above for semi-perfect obstruction theory now establish the Jiang-Thomas theory of signed Euler characteristic (cf. \cite{JiTh}) without any assumption on the derived geometry of $X$. 

\subsection{Dual obstruction cone}\label{S6.1}
In this subsection, we recall the dual obstruction cone by Jiang-Thomas \cite{JiTh}. Let $X$ be a \DM stack
equipped with a perfect obstruction theory
\beq\label{6.0}
E\lra \bbL_X
\eeq
whose obstruction sheaf is $Ob_X=h^1(E^\vee)$. 

\begin{defi}
The \emph{dual obstruction cone} of $X$ is defined as
\beq\label{6.1} 
N=\Spec_X(\mathrm{Sym} \cF)\mapright{\pi} X,\quad \cF=Ob_X
\eeq
which represents the functor
\[ \text{(Schemes over X)}\lra \text{(Sets)},\quad (S\mapright{f} X)\mapsto Hom_S(f^*\cF,\sO_S) \]
i.e. giving a morphism $S\to N$ over $X$ amounts to giving a morphism $S\mapright{f} X$ together with a homomorphism $f^*\cF\to \sO_S$.
\end{defi}
Finding a cosection $\sigma:\cF\to \sO_X$ of $\cF$ is the same as finding a section $s:X\to N$ of \eqref{6.1}.

\medskip

By a standard argument on perfect obstruction theory, there is an \'etale cover $\{X\lalp\to X\}$ of $X$ such that $X\lalp$ is the zero locus of a section $s\lalp$ of a trivial vector bundle $V\lalp$ on a smooth variety $A\lalp$ 
\beq\label{6.2}\xymatrix{
& V\lalp\ar[d]\\
X\lalp=\mathrm{zero}(s\lalp)\ar@{^(->}[r] & A\lalp
}\eeq
and that the perfect obstruction theory $E$ restricted to $X\lalp$ is 
\beq\label{6.3}
E|_{X\lalp}\cong [V\lalp^\vee|_{X\lalp}\mapright{(ds\lalp)^\vee} \Omega_{A\lalp}|_{X\lalp}] .
\eeq
Since the obstruction sheaf $Ob_X|_{X\lalp}$ of the perfect obstruction theory $E|_{X\lalp}$ is 
$\mathrm{coker}(ds\lalp)$, $N\lalp:=N|_{X\lalp}$ is the subscheme of $V\lalp^\vee|_{X\lalp}$
defined by the vanishing of the section 
$$(ds\lalp)^\vee: V\lalp^\vee|_{X\lalp}\lra \pi\lalp^*\Omega_{A\lalp}|_{X\lalp}$$
where $\pi\lalp:V\lalp|_{X\lalp}^\vee \to X\lalp$ is the projection. 

If we let $x_1,\cdots,x_n$ be coordinates for $A\lalp$ and $y_1,\cdots, y_r$ be coordinates of the fiber of $V\lalp^\vee\to A\lalp$, $N\lalp=N|_{X\lalp}$ is defined by the ideal generated by 
$$\{s_i\}_{1\le i\le n},\quad \{\sum_iy_i\frac{\partial s_i}{\partial x_j}\}_{1\le j\le r}$$
where $s_i=s_i(x_1,\cdots,x_n)$ are the coordinate functions of $s\lalp$. Let
\beq\label{6.4}
\tilde{s}\lalp:V\lalp^\vee\mapright{s\lalp} \CC
\eeq
be the function on the dual bundle $V\lalp^\vee$ of $V\lalp$, defined by the section $s\lalp$ of $V\lalp$ over $A\lalp$:
\[ \tilde{s}\lalp=\sum_{i=1}^r y_i s_i . \]
Hence the critical locus of $\tilde{s}\lalp$ is defined by the vanishing of 
\beq\label{6.9}
d\tilde{s}\lalp=\sum_i s_i dy_i +\sum_j (\sum_i y_i\frac{\partial s_i}{\partial x_j})dx_j
\eeq
and so $N\lalp$ is the critical locus of the function $\tilde{s}\lalp$ defined on the smooth variety $V\lalp^\vee$, i.e.
\beq\label{6.5} N\lalp=N|_{X\lalp} =\mathrm{Crit}(\tilde{s}\lalp) .\eeq
In particular, $N\lalp$ is equipped with a symmetric perfect obstruction theory
\beq\label{6.6}
F\lalp:=[T_{V\lalp^\vee}|_{N\lalp}\mapright{d(d\tilde{s}\lalp)} \Omega_{V\lalp^\vee}|_{N\lalp}].
\eeq

\medskip

\subsection{Semi-perfect obstruction theory of $N$}\label{S6.2}
In this subsection, we prove that the dual obstruction cone $N$ in \eqref{6.1} has a symmetric semi-perfect obstruction theory defined by \eqref{6.6}.

\begin{theo}\label{6.7}
The dual obstruction cone $N$ has a symmetric semi-perfect obstruction theory
\beq\label{6.8}
\{N\lalp\to N \}, \quad \{F\lalp\to \bbL_{N\lalp}\}
\eeq
where $F\lalp$ is defined by \eqref{6.6} and $N\lalp=N|_{X\lalp}=N\times_XX\lalp$. 
\end{theo}
 
\begin{proof}
We use the notation and local description of \S\ref{S6.1}. Since \eqref{6.6} is symmetric, the obstruction sheaves
\beq\label{6.10}
Ob_{N\lalp}=h^1(F\lalp^\vee)\cong \Omega_{N\lalp}=\Omega_N|_{N\lalp}
\eeq
glue to the cotangent sheaf $\Omega_N$. We have to show that the perfect obstruction theories
$F\lalp$ and $F_\beta$ on 
\beq\label{6.11}
N_\albe:=N\lalp\times_NN_\beta=N\times_XX_\albe
\eeq
give the same obstruction assignment (cf. \S\ref{sec2}). 

Let $x\in N_\albe$ be a closed point and let $B$ be an Artin local ring with maximal ideal $m_B$. 
Let $I$ be an ideal of $B$ satisfying $I\cdot m_B=0$. Let $\bar B=B/I$. Let $\bar g: \Spec\bar B\to N_\albe$ be a morphism that sends the closed point of $\Spec \bar B$ to $x$. We then have a diagram
\beq\label{6.12}\xymatrix{
\Spec \bar B\ar[r]^{\bar g}\ar@{^(->}[d] & N_\albe \ar[r]\ar[dr] & N\lalp\ar@{^(->}[r] & V\lalp^\vee\\
\Spec B && N_\beta \ar@{^(->}[r] & V_\beta^\vee
}\eeq
Since $V\lalp^\vee$ and $V^\vee_\beta$ are smooth, there are morphisms
\beq\label{6.13}
V^\vee\lalp\mapleft{g\lalp} \Spec B\mapright{g_\beta} V_\beta^\vee
\eeq
that extend \eqref{6.12}. The functions $\tilde{s}\lalp$ and $\tilde{s}_\beta$ in \eqref{6.4} give us sections
\beq\label{6.14}
d\tilde{s}\lalp:V\lalp^\vee\lra \Omega_{V\lalp^\vee},\quad d\tilde{s}_\beta:V_\beta^\vee\lra \Omega_{V_\beta^\vee}.
\eeq
Composing \eqref{6.13} and \eqref{6.14} with the canonical 
\beq\label{6.15}
\rho\lalp:I\otimes_\CC\Omega_{V\lalp^\vee}|_x\to I\otimes_\CC\Omega_{N\lalp}|_x, \quad
\rho_\beta: I\otimes_\CC\Omega_{V_\beta^\vee}|_x\to I\otimes_\CC \Omega_{N_\beta}|_x,
\eeq
we obtain the obstruction classes
\beq\label{6.16}
ob\lalp=\rho\lalp ( d\tilde{s}\lalp\circ g\lalp|_x)\in I\otimes_\CC\Omega_{N_\albe}|_x, \eeq
\[ob_\beta=\rho_\beta ( d\tilde{s}_\beta\circ g_\beta|_x)\in I\otimes_\CC\Omega_{N_\albe}|_x
\]
by \cite[Lemma 1.28]{KLp}. The theorem follows if we show that $ob\lalp=ob_\beta$.

Note that the perfect obstruction theory \eqref{6.6} fits into the commutative diagram of exact sequences
\beq\label{6.17}\xymatrix{
& 0\ar[d] &0\ar[d]&0\ar[d]&.\ar[d]\\
0\ar[r] & \pi^*Ob_{X\lalp}^\vee \ar[r]\ar[d] & [\pi^*V\lalp^\vee|_{X\lalp}\ar[r]^{(ds\lalp)^\vee} \ar[d]
& \pi^*\Omega_{A\lalp}|_{X\lalp}]\ar[d]\ar[r] & \pi^*\Omega_{X\lalp}\ar[d]\ar[r] & 0\\
0\ar[r] & T_{N\lalp} \ar[r]\ar[d] & [T_{V\lalp^\vee}|_{N\lalp}\ar[r]^{d(d\tilde{s}\lalp)} \ar[d]
& \Omega_{V\lalp^\vee}|_{N\lalp}]\ar[d]\ar[r] & \Omega_{N\lalp}\ar[d]\ar[r] & 0\\
0\ar[r] & \pi^*T_{X\lalp} \ar[r]\ar[d] & [\pi^*T_{A\lalp}|_{X\lalp}\ar[r]^{ds\lalp} \ar[d]
& \pi^*V\lalp|_{X\lalp}]\ar[d]\ar[r] & \pi^*Ob_{X\lalp}\ar[d]\ar[r] & 0\\
&.&0&0&0
}\eeq
where $\pi:N\lalp\to X\lalp$ is the projection. 
Hence we may write $$ob_\alpha=\hat{ob}\lalp+\bar{ob}\lalp \and ob_\beta=\hat{ob}_\beta+\bar{ob}_\beta,$$
$$\hat{ob}_\alpha, \hat{ob}_\beta\in I\otimes_\CC\pi^*Ob_{X_\albe}|_x,\and 
\bar{ob}_\alpha, \bar{ob}_\beta\in I\otimes_\CC\pi^*\Omega_{X_\albe}|_x.$$ where $\hat{ob}$ and $\bar{ob}$ correspond to the first and the second terms respectively of the right side of \eqref{6.9}. The theorem follows if we show that $\hat{ob}\lalp=\hat{ob}_\beta$ and $\bar{ob}\lalp=\bar{ob}_\beta$.

We first observe that $\hat{ob}\lalp=\hat{ob}_\beta$. This is because they are in fact the obstruction classes for $X\lalp$ and $X_\beta$ given by the morphisms
\[\xymatrix{
\Spec \bar B\ar[r]^{\bar g}\ar@{^(->}[d] & N_\albe\ar[r]^\pi & X_\albe \ar[r]\ar[dr] & X\lalp\ar@{^(->}[r] & A\lalp\\
\Spec B && &X_\beta \ar@{^(->}[r] & A_\beta
}\]
and their extensions
\[
A\lalp\longleftarrow V\lalp^\vee\mapleft{g\lalp} \Spec B\mapright{g_\beta} V_\beta^\vee \longrightarrow A_\beta .
\]
These classes coincide because they arise from a perfect obstruction theory \eqref{6.0}.

Next we show that $\bar{ob}\lalp=\bar{ob}_\beta$. 
Note that the last term of \eqref{6.9} is exactly the pullback of 
$$(ds\lalp)^\vee:V^\vee\lalp\to \Omega_{A\lalp}.$$
The obstruction class $\bar{ob}_\alpha$ is the image of the closed point in $\Spec \bar B$ by the map 
$$\Spec \bar B\mapright{\bar g} N_\albe\hookrightarrow V\lalp^\vee|_{X_\albe}\mapright{(ds\lalp)^\vee} \pi^*\Omega_{A\lalp}|_{X_\albe}\lra \pi^*\Omega_{X_\albe}$$ 
and twisting by $I$. By \eqref{6.3}, we have isomorphisms
\[
[V\lalp^\vee|_{X_\albe}\mapright{(ds\lalp)^\vee} \Omega_{A\lalp}|_{X_\albe}]\cong E|_{X_\albe}\cong
[V_\beta^\vee|_{X_\albe}\mapright{(ds_\beta)^\vee} \Omega_{A_\beta}|_{X_\albe}]
\]
that give rise to a commutative diagram
\beq\label{6.20}\xymatrix{
N_\albe \ar[r]\ar@{=}[d] & [\pi^*V\lalp^\vee|_{X_\albe}   \ar[r]^{(ds\lalp)^\vee}& \pi^*\Omega_{A\lalp}|_{X_\albe}  ]\ar[r] & \pi^*\Omega_{X_\albe}\ar@{=}[d]\\
N_\albe \ar[r]\ar@{=}[d] & [\pi^*E^{-1} \ar[d]\ar[u]  \ar[r]& \pi^*E^0\ar[u]\ar[d]  ]\ar[r] & \pi^*\Omega_{X_\albe}\ar@{=}[d]\\
N_\albe \ar[r] & [\pi^*V_\beta^\vee|_{X_\albe}   \ar[r]^{(ds_\beta)^\vee}& \pi^*\Omega_{A_\beta}|_{X_\albe}  ] \ar[r]& \pi^*\Omega_{X_\albe}
}\eeq
for some locally free sheaves $E^{-1}$ and $E^0$ on $X_\albe$.
The equality $\bar{ob}\lalp=\bar{ob}_\beta$ follows from the commutativity of \eqref{6.20}.
This proves the theorem.  
\end{proof}

We will see below that the semi-perfect obstruction theory \eqref{6.8} of the dual obstruction cone $N$ is $\CC^*$-equivariant and admits a cosection of the obstruction sheaf $Ob_N=\Omega_N$. So we may apply the torus localization and cosection localization proved in this paper.

\begin{ques}\label{6.21} Is the dual obstruction cone $N$ a critical virtual manifold in the sense of \cite[Definition 1.5]{KLp}?\end{ques}

\medskip

\subsection{Localization of the virtual cycle of the dual obstruction cone}\label{S6.3}

Let $X$ be a \emph{scheme} equipped with a perfect obstruction theory $E$ that 
gives us the virtual fundamental  class $[X]\virt$. We assume that $E$ admits a global resolution $[E^{-1}\to E^0]$ by locally free sheaves on $X$. Let $\cF=Ob_X=h^1(E^\vee)$. The natural grading on $\mathrm{Sym}_{\sO_X}\cF$ determines a $T=\CC^*$-action on $N=\Spec_X(\mathrm{Sym} \cF)$ whose fixed locus is precisely $X$. 

The action of $T$ on $N$ moreover gives us the Euler vector field $v$ on $N$ whose dual is a cosection
\[\sigma:\Omega_N\to \sO_N \]
which is surjective on $N-X$ by \cite[\S3]{JiTh}. 

By construction, the semi-perfect obstruction theoy of $N$ in Theorem \ref{6.7} is $T$-equivariant
and the obstruction sheaf is $Ob_N=\Omega_N$. By \eqref{6.17}, the virtual normal bundle of the $T$-fixed locus $X$ in $N$ is the dual $E^\vee$ of $E$ which admits a global resolution by assumption. Therefore, we can apply the localization theorems (Theorems \ref{3.0} and \ref{4.10}) that we proved for semi-perfect obstruction theory.

Jiang and Thomas in \cite{JiTh} considered five natural ways to define the notion of virtual signed Euler characteristic of $X$ as follows:
\begin{enumerate}
\item (Ciocan-Fontanine-Kapranov/Fantechi-G\"ottsche) $$e_1(X)=\int_{[X]\virt} c_{\mathrm{rk}E}(E^\vee),$$
\item (Graber-Pandharipande) $$e_2(X)=\int_{[X]\virt} \frac{c^T_{top}(E^0\otimes\mathbf{t})}{c^T_{top}(E^{-1}\otimes\mathbf{t})},$$
\item (Behrend) $$e_3(X)=e(N,\nu_N)=e(X,\nu_N|_X)$$ where $\nu_N$ is the Behrend function \cite{Behr},
\item (Kiem-Li) $$e_4(X)=\int_{[N]\virt_\loc} 1$$ for the cosection localized virtual cycle $[N]\virt_\loc\in A_0(X)$,
\item $e_5(X)=(-1)^{\mathrm{rk}E}e(X)$.
\end{enumerate}
By Theorems \ref{6.7}, \ref{3.0} and \ref{4.10}, we do not need to assume that $X$ comes from a quasi-smooth derived scheme as in \cite{JiTh}. So the following theorem of Jiang-Thomas holds without any assumption from derived geometry.
\begin{theo} \cite{JiTh}
Let $X$ be a scheme equipped with a perfect obstruction theory $E$ which admits a global resolution $[E^{-1}\to E^0]$ by locally free sheaves. Then the five definitions of virtual signed Euler characteristic are well defined and satisfy
$$e_1(X)=e_2(X),\quad e_3(X)=e_4(X)=e_5(X).$$
\end{theo}


\bibliographystyle{amsplain}

\end{document}